\documentclass[12pt]{article}
\usepackage{amsmath,amssymb,amsthm,cite,comment}

\newtheorem{theorem}{Theorem}[section]
\newtheorem{thmy}{Theorem}

\newtheorem{lemma}[theorem]{Lemma}
\newtheorem{corollary}[theorem]{Corollary}

\newcommand{\dd}{\displaystyle }

\def\barr{\begin{array}}
\def\earr{\end{array}}

\title{Two new functions related to the sum of element orders of a finite group}
\author{Marius T\u arn\u auceanu}
\date{August 7, 2026}

\begin{document}

\maketitle

\begin{abstract}
Let $G$ be a finite group and $H$ be a proper normal subgroup of $G$. In this note, we introduce two new functions $o(G,H)$ and $\psi''(G,H)$ involving the sums of element orders of $G$ and $H$. We prove that they satisfy certain inequalities, with equality if and only if $(G,H)$ is an equal order pair of a particular type.
\end{abstract}

{\small
\noindent
{\bf MSC2020\,:} Primary 20D60; Secondary 20E34.

\noindent
{\bf Key words\,:} element orders, finite groups, equal order pairs.}

\section{Introduction}

Let $G$ be a finite group. In 2009, H. Amiri, S.M. Jafarian Amiri and I.M. Isaacs introduced in their paper \cite{1} the function
\begin{equation}
\psi(G)=\sum_{x\in G}o(x),\nonumber
\end{equation}where $o(x)$ denotes the order of $x$ in $G$. They proved the following basic theorem.

\begin{thmy}
If $G$ is a finite group of order $n$, then $\psi(G)\leq\psi(C_n)$, and we have equality if and only if $G$ is cyclic.
\end{thmy}Since then many authors have studied the properties of the function $\psi(G)$ and its relations with the structure of $G$. Many other functions that are connected to $\psi(G)$ have been also studied, such as $o(G)=\frac{\psi(G)}{|G|}$ in \cite{11,12,13,14,16,22}, $\psi'(G)=\frac{\psi(G)}{\psi(C_{|G|})}$ in \cite{4,5,8,9,10,21} or $\psi''(G)=\frac{\psi(G)}{|G|^2}$ in \cite{20}.

Note that Theorem A follows from the next result.\newpage

\begin{thmy}
If $G$ is a finite group of order $n$, then there is a bijection $f:G\longrightarrow C_n$ such that $o(x)$ divides $o(f(x))$, for all $x\in G$.
\end{thmy}

This has been formulated as a question by I.M. Isaacs (see Problem 18.1 in \cite{17}) and proved for some particular groups by F. Ladisch \cite{15} and M. Amiri and S.M. Jafarian Amiri \cite{2}. A proof for arbitrary groups has been recently given by M. Amiri \cite{3}. 

In the current note, given a proper normal subgroup $H$ of $G$, we consider the functions
\begin{equation}
o(G,H)=\frac{\psi(G)-\psi(H)}{|G|-|H|} \mbox{ and } \psi''(G,H)=\frac{\psi(G)-\psi(H)}{|G|^2-|H|^2}\,.\nonumber
\end{equation}These relative invariants present a natural extension of the classical sum of element orders $\psi(G)$ and the average element order $o(G) = \frac{\psi(G)}{|G|}$. Specifically, notice that $\psi(G)-\psi(H)$ computes the sum of the orders of all elements in $G \setminus H$. Consequently, the invariant $o(G,H)$ represents the arithmetic mean of the orders of elements lying outside of $H$. When $H=1$, $o(G,1)=\frac{\psi(G)-1}{|G|-1}$ naturally reduces to the average order of the non-identity elements of $G$. Similarly, $\psi''(G,H)$ normalizes this sum over the difference of squares $|G|^2-|H|^2$, a scaling that provides finer structural control when investigating the boundary cases of Camina pairs and Frobenius groups.

Our main result is stated as follows.

\begin{theorem}
Let $G$ be a finite group and $H$ be a proper normal subgroup of $G$. Then
\begin{equation}
o(G,H)\geq o(G/H,1)
\end{equation}and
\begin{equation}
\psi''(G,H)\leq\psi''(G/H,1).
\end{equation}Assume that $H$ is non-trivial. Then we have equality 
in (1) if and only if $(G,H)$ is an equal order pair with $H$ an isolated subgroup, and in (2) if and only if $G\cong C_{p^n}$ with $p$ prime, $n\geq 2$, and $H=p^{n-i}G\cong C_{p^i}$, $i=1,\dots,n-1$, or $G\cong Q_{2^n}$ with $n\geq 3$ and $H=Z(G)\cong C_2$.
\end{theorem}

Recall that $(G,H)$ is called an \textit{equal order pair} if for every $x\in G\setminus H$ all the elements of the coset $xH$ have
the same order as $x$, and that $H$ is called an \textit{isolated subgroup} if for every $x\in G\setminus H$ we have $\langle x\rangle\cap H=1$. Equal order pairs have been recently introduced (see \cite{7}) as a generalization of Camina pairs. Theorem 1.1 reveals two particular types of such pairs that will be investigated in the first part of Section 2, while in the second part we will give the proofs of our main results.\newpage

The inequalities (1) and (2) can be used to obtain lower/upper bounds of the function $\psi(G)$ when we do not have an explicit formula. Many such examples are obtained when $G$ is an extension of a group $H$ by a group $K$ such that both $\psi(H)$ and $\psi(K)$ are known. A particular case is as follows.

\begin{corollary}
Let $G$ is a finite metabelian group and $H$ be an abelian normal subgroup of $G$ such that $G/H$ is abelian. Then
\begin{equation}
\psi(H)+|H|(\psi(G/H)-1)\leq\psi(G)\leq\psi(H)+|H|^2(\psi(G/H)-1).\nonumber
\end{equation}
\end{corollary}

We exemplify this case for ZM-groups, i.e. for finite groups with all Sylow subgroups cyclic. Such a group is of type
\begin{equation}
{\rm ZM}(m,n,r)=\langle a, b \mid a^m = b^n = 1,
\hspace{1mm}b^{-1} a b = a^r\rangle, \nonumber
\end{equation}where the triple $(m,n,r)$ satisfies the conditions
\begin{equation}
{\rm gcd}(m,n)={\rm gcd}(m,r-1)=1 \mbox{ and } r^n
\equiv 1 \hspace{1mm}({\rm mod}\hspace{1mm}m). \nonumber
\end{equation}Clearly, $|{\rm ZM}(m,n,r)|=mn$ and ${\rm ZM}(m,n,r)$ is a certain semidirect product of $C_m$ by $C_n$. Note that $Z({\rm ZM}(m,n,r))=\langle b^d\rangle$, where $d$ is the multiplicative order of $r$ modulo $m$. The subgroups of ${\rm ZM}(m,n,r)$ are as follow. Set
\begin{equation}
L=\left\{(m_1,n_1,s)\in\mathbb{N}^3 \hspace{1mm}\mid\hspace{1mm}
m_1|m,\hspace{1mm} n_1|n,\hspace{1mm} s<m_1,\hspace{1mm}
m_1|s\frac{r^n-1}{r^{n_1}-1}\right\}.\nonumber
\end{equation}Then there is a bijection between $L$ and the subgroup lattice
$L({\rm ZM}(m,n,r))$ of ${\rm ZM}(m,n,r)$, namely the function
that maps a triple $(m_1,n_1,s)\in L$ into the subgroup $H_{(m_1,n_1,s)}$ defined by\vspace{-1mm}
\begin{equation}
H_{(m_1,n_1,s)}=\bigcup_{k=1}^{\frac{n}{n_1}}\alpha(n_1,
s)^k\langle a^{m_1}\rangle=\langle a^{m_1},\alpha(n_1, s)\rangle,\vspace{-1mm}\nonumber
\end{equation}where $\alpha(x, y)=b^xa^y$, for all $0\leq x<n$ and $0\leq y<m$.
Note that:
\begin{itemize}
\item[-] $|H_{(m_1,n_1,s)}|=\frac{mn}{m_1n_1}$\,, for any $s$ satisfying $(m_1,n_1,s)\in L$;
\item[-] $H_{(m_1,n_1,s)}$ is cyclic if and only if $\frac{m}{m_1}\mid r^{n_1}-1$;
\item[-] two subgroups of ${\rm ZM}(m,n,r)$ are conjugate if and only if they have the same order.
\end{itemize}

\begin{corollary}
Under the above notation, we have
\begin{equation}
\psi(C_m)+m(\psi(C_n)-1)\leq\psi({\rm ZM}(m,n,r))<\psi(C_m)+m^2(\psi(C_n)-1).\nonumber
\end{equation}Moreover, the first inequality becomes an equality if and only if $d=n$ and for all proper divisors $m_1$ of $m$ and $n_1$ of $n$, we have $m_1\nmid r^{n_1}-1$.\footnote{Note that this happens, for example, when $m$ or $n$ are primes.}
\end{corollary}

Finally, we formulate a natural question related to our main result.

\bigskip\noindent{\bf Open problem.} Study equal order pairs $(G,H)$ where $H$ is an isolated subgroup. 

As shown in Theorem 1.1, such pairs arise precisely when equality holds in (1). Therefore, investigating their properties and structural classification constitutes a direct and meaningful extension of the present work.
\bigskip

Most of our notation is standard and will usually not be repeated here. Elementary notions and results on groups can be found in \cite{18}.

\section{Proofs of the main results}

First of all, we study equal order pairs $(G,H)$ satisfying
\begin{equation}
o(x)=o(xH), \forall\, x\in G\setminus H,
\end{equation}and
\begin{equation}
o(x)=o(xH)|H|, \forall\, x\in G\setminus H,
\end{equation}respectively. These correspond to the cases when we have equality in the inequalities (1) and (2).

We observe that the condition (3) is equivalent to
\begin{equation}
\langle x\rangle\cap H=1, \forall\, x\in G\setminus H,\nonumber
\end{equation}i.e. to the fact that $H$ is an isolated subgroup. Clearly, $(C_p\times C_p,C_p\times 0)$ with $p$ prime is an equal order pair satisfying (3), while $(D_8,Z(D_8))$ and $(Q_8,Z(Q_8))$ are equal order pairs that does not satisfy (3).

An interesting example of such a pair is given by the following lemma. 

\begin{lemma}
Let $G={\rm ZM}(m,n,r)$ and $H=\langle a\rangle$. Denote by $d$ the multiplicative order of $r$ modulo $m$. Then the following conditions are equivalent: 
\begin{itemize}
\item[{\rm a)}] $(G,H)$ is an equal order pair.
\item[{\rm b)}] $(G,H)$ is a Camina pair.
\item[{\rm c)}] $G$ is a Frobenius group with kernel $H$.
\item[{\rm d)}] $H$ is an isolated subgroup of $G$.
\item[{\rm e)}] $d=n$ and $m_1\nmid r^{n_1}-1$ for all proper divisors $m_1$ of $m$ and $n_1$ of $n$.
\end{itemize}
\end{lemma}
\begin{proof}
\begin{description}
\item[{\rm a)} $\Leftrightarrow$ {\rm b)}] Assume that $(G,H)$ is an equal order pair and let $b^xa^y\in G$, where $1\leq x<n$ and $0\leq y<m$. Since $b^xa^y\not\in H$, we have $o(b^xa^y)=o(b^x)$, which implies that $|\langle b^xa^y\rangle|=|\langle b^x\rangle|$. Then there are $0\leq u<n$ and $0\leq v<m$ such that $\langle b^xa^y\rangle=\langle b^x\rangle^{\alpha(u,v)}$. This leads to $b^xa^y=(b^{xz})^{\alpha(u,v)}$ for some $0\leq z<o(b^x)$ with $(z,o(b^x))=1$. It follows that
\begin{equation}
b^xa^y=b^{xz}a^{v(1-r^{xz})}\nonumber    
\end{equation}and so $b^x=b^{xz}$, i.e. $z=1$. Then $b^xa^y$ is conjugate to $b^x$ and thus $(G,H)$ is a Camina pair.
\item[{\rm b)} $\Leftrightarrow$ {\rm c)}] It follows by \cite{6}, since $|H|$ and $[G:H]$ are coprime.
\item[{\rm c)} $\Leftrightarrow$ {\rm d)}] It suffices to observe that if there is $1\leq y<m$ with
\begin{equation}
a^y\in\langle b^ua^v\rangle\cap H \mbox{ for some } 1\leq u<n \mbox{ and } 0\leq v<m,\nonumber 
\end{equation}then 
\begin{equation}
b^ua^v\in C_G(a^y)\leq H,\nonumber 
\end{equation}a contradiction.
\item[{\rm d)} $\Leftrightarrow$ {\rm e)}] For every triple $(m_1,n_1,s)\in L$, we have $H_{(m_1,n_1,s)}\cap H=\langle a^{m_1}\rangle$. Then $H$ is an isolated subgroup of $G$ if and only if $H_{(m_1,n_1,s)}\cap H=1$ for all triples $(m_1,n_1,s)\in L$ satisfying $n_1\neq n$ and $\frac{m}{m_1}\mid r^{n_1}-1$, i.e. if and only if
\begin{equation}
n_1\neq n \mbox{ and } \frac{m}{m_1}\mid r^{n_1}-1 \mbox{ imply that } m_1=m.\nonumber 
\end{equation}Obviously, this is equivalent to the condition e).
\end{description}
\end{proof}

Next we describe equal order pairs $(G,H)$ with the property (4). 

\begin{lemma}
An equal order pair $(G,H)$ satisfies (4) if and only if either $G\cong C_{p^n}$ with $p$ prime, $n\geq 2$, and $H=p^{n-i}G\cong C_{p^i}$, $i=1,\dots,n-1$, or $G\cong Q_{2^n}$ with $n\geq 3$ and $H=Z(G)\cong C_2$.\footnote{Note that these groups $G$ are exactly the finite groups having a unique minimal subgroup - see e.g. (4.4) of \cite{18}, II.}
\end{lemma}
\begin{proof}
Assume that the equal order pair $(G,H)$ satisfies (4). Let $x\in G\setminus H$ and denote $n=o(xH)$. Then 
\begin{equation}
o(x)=o(xH)o(x^n)\nonumber
\end{equation}and so $o(x^n)=|H|$, implying that $H=\langle x^n\rangle\subseteq\langle x\rangle$. Thus $H$ is a breaking point in the poset of cyclic subgroups of $G$, i.e. 
\begin{equation}
\mbox{for every } X\in C(G), \mbox{ we have } X\leq H \mbox{ or } H\leq X.\nonumber
\end{equation}By Theorem 1.1 of \cite{19}, it follows that $G$ is either a cyclic $p$-group of order at least $p^2$ or a generalized quaternion $2$-group. In the first case $H$ can be any proper non-trivial subgroup of $G$, while in the second case $H$ must be the unique subgroup of order $2$ of $G$.

The converse is obvious.
\end{proof}

We are now able to prove Theorem 1.1 and its consequences.

\begin{proof}[Proof of Theorem 1.1.] 
Let $m=[G:H]$ and $G/H=\{x_1H=H,x_2H,\dots,x_mH\}$. Then we have
\begin{equation}
\psi(G)=\psi(H)+\dd\sum_{i=2}^m\psi(x_iH)=\psi(H)+\dd\sum_{i=2}^m\dd\sum_{h\in H}o(x_ih).
\end{equation}For every $i=2,\dots,m$, let $y_i,z_i\in x_iH$ such that
\begin{equation}
o(y_i)=\min\{o(x_ih)\mid h\in H\} \mbox{ and } o(z_i)=\max\{o(x_ih)\mid h\in H\}.\nonumber
\end{equation}We obtain
\begin{equation}
o(y_i)|H|\leq\dd\sum_{h\in H}o(x_ih)\leq o(z_i)|H|, \forall\, i=2,\dots,m\nonumber
\end{equation}and so
\begin{equation}
\dd\sum_{i=2}^mo(y_i)\leq\dd\frac{\psi(G)-\psi(H)}{|H|}\leq\dd\sum_{i=2}^mo(z_i)
\end{equation}by (5). Since 
\begin{equation}
o(y_iH)\leq o(y_i) \mbox{ and } o(z_i)\leq o(z_iH)|H|, \forall\, i=2,\dots,m,\nonumber
\end{equation}(6) leads to
\begin{equation}
\psi(G/H)-1=\sum_{i=2}^m\o(y_iH)\leq\dd\frac{\psi(G)-\psi(H)}{|H|}\leq|H|\sum_{i=2}^mo(z_iH)=|H|(\psi(G/H)-1),\nonumber
\end{equation}which means
\begin{equation}
o(G,H)\geq o(G/H,1) \mbox{ and } \psi''(G,H)\leq\psi''(G/H,1),\nonumber
\end{equation}as desired.

Moreover, $o(G,H)=o(G/H,1)$ if and only if for every $i=2,\dots,m$, we have
\begin{equation}
o(y_i)=o(y_ih),\, \forall\, h\in H, \mbox{ and } o(y_i)=o(y_iH).\nonumber
\end{equation}Since each $y\in G\setminus H$ belongs to a coset $y_iH$ for some $i\in\{2,\dots,m\}$, it is easy to see that the above conditions are equivalent to
\begin{equation}
o(y)=o(yh),\, \forall\, h\in H, \mbox{ and } o(y)=o(yH).\nonumber
\end{equation}Thus $(G,H)$ is an equal order pair satisfying (3).

Similarly, $\psi''(G,H)=\psi''(G/H,1)$ if and only if for every $i=2,\dots,m$, we have
\begin{equation}
o(z_i)=o(z_ih),\, \forall\, h\in H, \mbox{ and } o(z_i)=o(z_iH)|H|,\nonumber
\end{equation}or equivalently for every $z\in G\setminus H$,
\begin{equation}
o(z)=o(zh),\, \forall\, h\in H, \mbox{ and } o(z)=o(zH)|H|,\nonumber
\end{equation}that is $(G,H)$ is an equal order pair satisfying (4). Now the conclusion is obtained from Lemma 2.2, completing the proof.
\end{proof}

\begin{proof}[Proof of Corollary 1.2.] 
From inequality (1), we have
\begin{equation}
\frac{\psi(G)-\psi(H)}{|G|-|H|}\geq\frac{\psi(G/H)-1}{|G/H|-1}\,,\nonumber
\end{equation}that is,
\begin{equation}
\psi(G)\geq\psi(H)+|H|(\psi(G/H)-1).\nonumber
\end{equation}Similarly, from inequality (2), we obtain
\begin{equation}
\frac{\psi(G)-\psi(H)}{|G|^2-|H|^2}\leq\frac{\psi(G/H)-1}{|G/H|^2-1}\,,\nonumber
\end{equation}that is,
\begin{equation}
\psi(G)\leq\psi(H)+|H|^2(\psi(G/H)-1).\nonumber
\end{equation}
\end{proof}

\begin{proof}[Proof of Corollary 1.3.] 
The inequalities are obtained by applying Corollary 1.2 to $G={\rm ZM}(m,n,r)$ and $H=\langle a\rangle\cong C_m$. Note that the second one cannot be an equality because a ZM-group is not a $p$-group and therefore it cannot have a unique minimal subgroup. The first one become an equality if and only if $(G,H)$ is an equal order pair with $H$ an isolated subgroup, and the conclusion follows by Lemma 2.1. 
\end{proof}

\bigskip\noindent{\bf Acknowledgements.} The author is grateful to the reviewer for remarks which improve the previous version of the paper.
\bigskip

\bigskip\noindent{\bf Funding.} The author did not receive support from any organization for the submitted work.

\bigskip\noindent{\bf Conflicts of interests.} The author declares that he has no conflict of interest.

\bigskip\noindent{\bf Data availability statement.} My manuscript has no associated data.

\vspace*{3ex}\small

\hfill
\begin{minipage}[t]{5cm}
Marius T\u arn\u auceanu \\
Faculty of  Mathematics \\
``Al.I. Cuza'' University \\
Ia\c si, Romania \\
e-mail: {\tt tarnauc@uaic.ro}
\end{minipage}

\end{document}